\documentclass{amsart}[13pt]
\usepackage{amssymb, latexsym, amsmath, amsthm}
\newtheorem{theorem}{Theorem}
\newtheorem{lemma}[theorem]{Lemma}
\newtheorem{corollary}[theorem]{Corollary}
\newtheorem{proposition}[theorem]{Proposition}

\makeatletter
\@namedef{subjclassname@2010}{%
  \textup{2010} Mathematics Subject Classification}
\makeatother

\begin{document}

%Revision 1

%  submitted to Acta Mathematica Hungarica, Fall 2017
%2016 IF 0.583
%referece number AMH/17/518

%maybe change this on p4: %If $p<q$, then a weakly-$q$-$L$-set is a weakly-$p$-$L$-set, since $\ell_p^w(X)\subseteq \ell_q^w(X)$.

% never forget the Introduction to the paper; separate from Definitions and Notation

\title{The $p$-Gelfand Phillips Property in Spaces of Operators and Dunford-Pettis like sets}

\author{Ioana Ghenciu}
\address{ Mathematics Department\\
 University of Wisconsin-River Falls\\
   Wisconsin, 54022}
\email{ioana.ghenciu@uwrf.edu}



\maketitle

\begin{abstract}
The $p$-Gelfand Phillips property ($1\le p<\infty$) is studied in spaces of operators. 
Dunford - Pettis type like sets are studied in  Banach spaces. We discuss Banach spaces $X$ with the  property that every $p$-convergent operator $T:X\to Y$ is weakly compact, for every Banach space $Y$. 
%We use operators with p-convergent adjoints to give a sufficient condition for an operator to be weakly precompact. 
\end{abstract}

% individual as well as in spaces of operators. %$W(X,\ell_\infty)$ $C_p(X,\ell_\infty)$

%\begin{abstract}
%We study equivalent conditions for the Dunford-Pettis property of order $p$, $1<p<\infty$. We introduce the  weakly-p-L-subsets of $X^*$ and weakly-p-Dunford-Pettis subsets of $X$. We give a characterization of weakly-p-L-sets using weakly-p-precompact  operators. We say that the   space $X$ has the reciprocal Dunford-Pettis of order p  ($RDP_p$) property if every weakly-p-L subset of $X^*$ is relatively weakly compact and that $X$ has  property $RDP^*_p$  if every weakly-p-Dunford Pettis subset of $X$ is relatively weakly compact. We study Banach spaces with properties $RDP_p$ and $RDP^*_p$. We also study the complementability of the space of weakly compact operators $W(X,\ell_\infty)$ in the space of weakly p-convergent operators $C_p(X,\ell_\infty)$. 
%\end{abstract}

%every p-convergent  operator $T$ from $X$ to any Banach space $Y$ is weakly compact 
%and weakly-p-Dunford Pettis sets using weakly precompact  operators. 
%\section{Introduction}

 \section{Introduction} 
\vspace{0.1in}

% I should have written: 
%We study \emph{weakly-p-Dunford-Pettis sets} and  \emph{weakly-p-L-sets} ($1\le p< \infty$).  (instead of we introduce and study; they were introduced in a preprint of mine sent to Advances Of Operator Theory)

% These results are motivated by results in \cite{GLB} and \cite{CS}. 

Numerous papers have investigated whether spaces of operators inherit the  Gelfand-Phillips property when the co-domain and the dual of the domain possess the respective property; e.g., see  \cite{DE},   \cite{EJMPA}, \cite{GAFA}, and \cite{SM}. 
In \cite{FZ2} the authors introduced the $p$-Gelfand-Phillips property ($1\le p<\infty$), a property which is in general weaker than the Gelfand-Phillips property. 
In this paper limited $p$-convergent evaluation operators are used  to study the $p$-Gelfand-Phillips property in spaces of operators. 
 We show that if  $Y$ has the Gelfand-Phillips property and $M$ is a closed subspace of $K_{w^*}(X^*,Y)$ such that  the evaluation operator $\psi_{y^*}:M\to X$ is limited $p$-convergent  for each $y^*\in Y^*$,  then $M$ has the $p$-Gelfand-Phillip property.
%Next we give some applications to the spaces $\ell_1[X]$,   $L_1(\mu)\otimes_\epsilon X $, and $c_0(X)$. 
We prove that if $X^*$  has the $p$-Gelfand-Phillip property and $Y$ has the Schur property, and  $M$ is a closed subspace of $L(X,Y)$, then $M$ has the $p$-Gelfand-Phillip property  ($1\le p<\infty$). We also prove that if  $L_{w^*}(X^*,Y)$ has  the $p$-Gelfand-Phillip property, 
%Gelfand-Phillip  (resp. $p$-Gelfand-Phillip)
%then $X$ and $Y$ have  property $GP$ (resp. $p$-$GP$)  and either  $\ell_2 \not \hookrightarrow X$ or $\ell_2 \not \hookrightarrow Y$.
then at least one of the spaces $X$ and $Y$ does not contain $\ell_2$. 
%We also give necessary conditions for the space  $L_{w^*}(X^*,Y)$ to have  property  $p$-$GP$ ($1\le p<\infty$).

%We use operators with p-convergent adjoints to give a sufficient condition for an operator to be weakly precompact. 
%If $T:Y\to X$ is an operator such that $T^*:X^*\to Y^*$ is p-convergent, then $T$ is weakly precompact.

We  study \emph{weakly-p-Dunford-Pettis sets} and  \emph{weakly-p-L-sets} ($1\le p< \infty$).  
We show that  every  operator with $p$-convergent adjoint  is weakly precompact and that weakly-$p$-Dunford-Pettis sets are weakly precompact (for $2< p< \infty$).
%We show that weakly-p-Dunford-Pettis sets are weakly precompact (for $1< p< \infty$).
	% We give characterizations of weakly-p-L-sets and weakly-p-Dunford-Pettis sets.  
	We prove that a  bounded subset $A$ of $X^*$  is an weakly-$p$-$L$-subset of $X^*$ if and only if $T^*(A)$ is relatively compact  whenever $Y$ is a Banach space and  $T:Y \to X$ is a  weakly-$p$-precompact operator ($1\le p< \infty$). 
	%This result is motivated by a characterization of $L$-sets in \cite{IGMonath}.  
	
 %introduce and
	We  study two Banach space properties, called the \emph{reciprocal Dunford-Pettis property of order p} (or  \emph{property} $RDP_p$) and \emph{property} $RDP^*_p$ ($1\le p< \infty$).

 %We study Banach spaces with properties $RDP_p$ and $RDP^*_p$. 
We prove that a Banach space $X$ has property $RDP_p$ ($1\le p< \infty$) if and only if for every Banach space $Y$, every $p$-convergent operator $T:X\to Y$  is weakly compact. 
 We also prove that a Banach space $X$ has property $RDP^*_p$ ($1\le p< \infty$) if and only if for every Banach space $Y$, every operator $T:Y\to X$  with  $p$-convergent adjoint   is weakly compact. These results are motivated by results in \cite{CS}. 
% and \cite{GLB}. 

%We use some results of Lohman and Gonzalez-Onieva to study some isomorphic properties  and some lifting properties of Banach spaces. 

\section{Definitions and Notation} 
\vspace{0.1in}

Throughout this paper, $X$, $Y$, $E$ and $F$ will denote Banach spaces. The unit ball of $X$ will be denoted by $B_X$    and $X^*$ will denote the continuous linear dual of $X$. 
%, the closed linear span of a sequence $(x_n)$ in $X$ will be denoted by $[x_n]$, 
The space $X$ embeds in $Y$ (in symbols $X\hookrightarrow Y$) if $X$ is isomorphic to a closed subspace of $Y$. 
An operator $T: X \to Y$ will be a continuous and linear function. 
%The space of all operators from $X$ to $Y$ will be denoted by $L(X,Y)$, and the subspace of compact operators will be denoted by $K(X,Y)$. 
The set of all operators, weakly compact operators, and compact operators from $X$ to $Y$ will be denoted by $L(X,Y)$, $W(X,Y)$, and $K(X,Y)$. 
The $w^* - w$ continuous (resp. compact) operators from $X^*$ to $Y$ will be denoted by $L_{w^*}(X^*, Y)$ (resp.  $K_{w^*}(X^*, Y)$). 
The injective tensor product of two Banach spaces $X$ and $Y$ will be denoted by $X\otimes_\epsilon Y$.
The space $X\otimes_\epsilon Y$ can be embedded into the space $K_{w^*}(X^*,Y)$, by identifying $x\otimes y$ with the rank one operator $x^*\to \langle x^*,x\rangle \, y$ (see \cite{DU} for the theory of tensor products).

%bounded 

A subset $S$ of $X$ is said to be \emph{weakly precompact} provided that every  sequence from $S$ has a weakly Cauchy subsequence. An operator   $T:X\to Y$ is called \emph{ weakly precompact (or almost weakly compact)} if $T(B_X)$ is weakly precompact. 
%A Banach space $X$ is called \emph{weakly sequentially complete} if every weakly Cauchy sequence in $X$ is weakly convergent. 
	
	An operator $T:X\to Y$ is called \emph{completely continuous (or Dunford-Pettis)} if $T$ maps weakly convergent sequences to norm convergent sequences. 

A Banach space $X$ has the \emph{Dunford-Pettis property (DPP)} if every weakly compact operator $T:X \to Y$ is completely continuous, for any Banach space $Y$. If $X$ is a $C(K)$-space or an $L_1$-space, then $X$ has the $DPP$. The reader can check \cite{JDSS} and \cite{DU} for results related to the $DPP$.
%, \cite{JDSDP},

For $1\le p<\infty$, $p^*$ denotes the conjugate  of $p$. If $p=1$, $\ell_{p^*}$ plays the role of $c_0$. The unit vector basis of $\ell_p$ will be denoted by $(e_n)$. 
 %Let $1\le p<\infty$. A sequence $(x_n)$ in $X$  is called  \emph{(strongly) p-summable}  if $(\|x_n\|)\in \ell_p$ \cite[p. 32]{DJT}. Let $\ell_p(X)$ denote the  set of all  $p$-summable sequences in $X$ with the norm 
%$$ \|(x_n)\|_p=(\sum_{n=1}^\infty \| x_n \|^p )^{1/p}.$$

%, or equivalently, if there is a constant $C>0$ such that  
%$$ \sup_n \|\sum_{k=1}^n b_k x_k\|\le C \|(b_n)\|_{\ell_{p^*}}, $$
%for any sequence $(b_n)\in \ell_{p^*}$ (\cite{}). 

Let $1\le p\le \infty$. A sequence $(x_n)$ in $X$  is called  \emph{weakly p-summable}  if $(\langle x^*, x_n\rangle )\in \ell_p$ for each $x^*\in X^*$ \cite[p. 32]{DJT}. 
Let $\ell_p^w(X)$ denote the  set of all weakly $p$-summable sequences in $X$. The space $\ell_p^w(X)$ is a Banach space with the norm 
$$ \|(x_n)\|_p^w=\sup \{( \sum_{n=1}^\infty |\langle x^*, x_n \rangle |^p )^{1/p}: x^*\in B_{X^*}\}$$
%The spaces $L(\ell_{p^*},X)$ and $\ell_p^w(X)$ are isometrically isomorphic by the correspondence  $T\to (T(e_n))$ \cite{GroSP}. 

%Let $1\le p\le \infty$. A weakly p-sequence $(x_n)$ in $X$  is  \emph{unconditionally p-summable}  if 
%$$ \sup\{ (\sum_{n=m}^\infty |x^*(x_n)|^p)^{1/p}: x^*\in B_{X^*} \} \to 0 \quad \text{as} \quad m\to \infty.  $$
%Let $\ell_p^u(X)$ denote the  set of all unconditionally $p$-summable sequences in $X$. The sequence $(x_n)$ in $X$  is  unconditionally $1$-summable if and only if it is unconditionally summable.

%i.e. the series $\sum x_n$ is unconditionally convergent. 
%The spaces $K(\ell_{p^*},X)$ and $\ell_p^w(X)$ are isometrically isomorphic by the correspondence  $T\to (T(e_n))$ \cite{FS}. 

%1) \\ \noindent
We recall the following  isometries: 
 $L(\ell_{p^*},X) \simeq \ell_p^w(X)$ for $1<p<\infty$; $ L(c_0,X)\simeq \ell_p^w(X)  $ for $p=1$; 
$T\to (T(e_n))$ \cite{GroSP},  \cite[Proposition 2.2, p. 36]{DJT}.

% (\cite[Proposition 2.2]{DJT})
%\\2) $K(\ell_{p^*},X) \simeq \ell_p^u(X)$, ($T\to (T(e_n))$ \cite[Theorem 1.4]{FS}.)

A series $\sum x_n$ in $X$ is said to be \emph{weakly unconditionally convergent (wuc)} if for every $x^*\in X^*$, the series $\sum |x^*(x_n)|$ is convergent. An operator $T:X\to Y$ is  \emph{unconditionally converging} if it maps weakly unconditionally convergent series  to unconditionally convergent ones. 
%(or weakly 1-summable sequences to unconditionally 1-summable sequences)

%For $p=\infty$, $\ell^w_\infty(X)=\ell_\infty(X)$. 

Let $1\le p\le \infty$. An operator $T:X\to Y$ is called \emph{p-convergent} if $T$ maps weakly $p$-summable sequences into norm null sequences. The set of all $p$-convergent operators is denoted by $C_p(X,Y)$ \cite{CS}. 

The $1$-convergent operators are precisely the unconditionally converging operators and the $\infty$-convergent operators are precisely the completely continuous operators. 
If $p<q$, then $C_q(X,Y)\subseteq C_p(X,Y)$. 
 %Let $1\le p\le \infty$. A Banach space $X$ has the \emph{Dunford-Pettis property of order p}  $(DPP_p)$ ($1\le p\le \infty$) if every weakly compact operator $T:X \to Y$ is $p$-convergent, for any Banach space $Y$ (\cite{CS}). 

	A sequence $(x_n)$ in $X$ is called \emph{weakly-p-convergent} to $x\in X$ if the sequence $(x_{n}-x)$ is weakly $p$-summable \cite{CS}. The weakly-$\infty$-convergent sequences are precisely the weakly convergent sequences. 
	%, i.e. $(x_n-x)\in \ell_p^w(X)$. 
	
	Let $1\le p\le \infty$. A bounded subset $K$ of $X$ is \emph{relatively weakly-$p$-compact} (resp. \emph{ weakly-p-compact}) if every sequence in $K$ has a weakly-$p$-convergent subsequence with limit in $X$ (resp. in $K$). An operator $T:X\to Y$ is \emph{weakly-$p$-compact} if $T(B_X)$ is relatively weakly-$p$-compact \cite{CS}.
The set of weakly-$p$-compact operators $T:X\to Y$ will be denoted by $W_p(X,Y)$.

 If $p<q$, then $W_p(X,Y)\subseteq W_q(X,Y)$. A Banach space $X\in C_p$ (resp. $X\in W_p$) if $id(X)\in C_p(X,X)$  (resp. $id(X)\in W_p(X,X)$) \cite{CS}, where $id(X)$ is the identity map on $X$.
	
	Let $1\le p\le \infty$. A sequence $(x_n)$ in $X$ is called \emph{weakly-p-Cauchy} if $(x_{n_k}-x_{m_k})$ is weakly $p$-summable for any increasing sequences $(n_k)$ and $(m_k)$ of positive integers. 
		%\cite{CCDL}. 
	
	% \cite{}
	Every weakly-$p$-convergent sequence is weakly-$p$-Cauchy, and the weakly $\infty$-Cauchy sequences are precisely the weakly Cauchy sequences.  
	
 	Let $1\le p\le \infty$. We say that a subset $S$ of $X$ is called \emph{weakly-$p$-precompact} if  every  sequence from $S$ has a weakly-$p$-Cauchy subsequence. The weakly-$\infty$-precompact sets are precisely the weakly precompact sets. 
	% If a set is w-p-orecompact, then it is weakly precompact, thus bounded.
	
	%(or almost weakly-p-compact)
Let $1\le p\le \infty$.	An operator   $T:X\to Y$ is called \emph{ weakly-$p$-precompact  } if $T(B_X)$ is weakly-$p$-precompact. The set of all weakly-$p$-precompact operators $T:X\to Y$ is denoted by $WPC_p(X,Y)$. We say that $X\in WPC_p$ if $id(X)\in WPC_p(X,X)$. 

The weakly-$\infty$-precompact operators are precisely the weakly precompact operators. If $p<q$, then $\ell_p^w(X)\subseteq \ell_q^w(X)$, thus $WPC_p(X,Y) \subseteq WPC_q(X,Y)$. 
% WPC_p(X,Y)\subseteq WPC(X,Y)

%%The survey article by Diestel \cite{JDS} is an excellent source of information about classical contributions to the study of the DPP.

%A Banach space $X$ has the \emph{Dunford-Pettis property (DPP)} if every weakly compact operator $T:X \to Y$ is completely continuous, for any Banach space $Y$.  Equivalently, $X$ has the $DPP$ if and only if $x_n^* (x_n) \to 0$ whenever $(x_n^*)$ is weakly null in $X^*$ and $(x_n)$ is weakly null in $X$ \cite[Theorem 1]{JDSDP}.
%If $X$ is a $C(K)$-space or an $L_1$-space, then $X$ has the $DPP$. The reader can check \cite{JDSS}, \cite{JDSDP}, and \cite{DU} for results related to the $DPP$. 

	The bounded subset $A$ of $X$ is called a \emph{Dunford-Pettis} (resp. \emph{limited})  subset  of $X$ if each weakly null (resp. $w^*$-null) sequence $(x_n^*)$ in $X^*$ tends to $0$ uniformly on $A$; i.e.
$$
\sup_{x \in A} |x_n^* (x) | \to  0.
$$
Every DP (resp. limited) subset of $X$ is weakly precompact \cite{KA}, \cite[p. 377]{HR} (resp. \cite{BD}, \cite{TS}).
%; see also \cite{GLB}.
% i.e., if $S$ is a $DP$ subset of $X$ and $(x_n)$ is a sequence from $S$, then $(x_n)$ has a weakly Cauchy subsequence.  See \cite{KA} and \cite{HR}, p. 377, for proofs.
%The subset $A$ of $X$ is a $DP$ subset of $X$ if and only if $T(A)$ is relatively compact whenever $T: X \to Y$ is a weakly compact operator \cite{KA} if and only if $T(A)$ is relatively compact whenever $T: X \to Y$ is  an operator with  weakly precompact adjoint \cite{IGMonath}.
	%The bounded subset $A$ of $X$ is called a \emph{limited}  subset  of $X$ if each $w^*$-null sequence $(x_n^*)$ in $X^*$ tends to $0$ uniformly on $A$.
%Every  limited subset of $X$ is weakly precompact  \cite{BD}, \cite{TS}.

The sequence $(x_n)$ in $X$ is called  limited if the corresponding set of its terms is a limited set. If the sequence $(x_n)$ is also weakly null (resp.  weakly $p$-summable), then $(x_n)$ is called a limited weakly null (resp.  limited weakly $p$-summable) sequence in $X$. 

The space $X$ has the \emph{Gelfand-Phillips (GP) property} (or is a \emph{Gelfand-Phillips space}) if every limited subset of $X$ is relatively compact. 

%The following spaces have the Gelfand-Phillips property:
Banach spaces having  the Gelfand-Phillips property include, among others,
 Schur spaces, separably complemented spaces, spaces with $w^*$-sequential compact dual unit balls, separable spaces,  reflexive spaces,   spaces whose duals do not contain $\ell_1$,  and dual spaces $X^*$ whith $X$  not containing $\ell_1$    (\cite{BD},  \cite{EBPA}, \cite[p. 31]{TS}). 

%subspaces of weakly compactly generated spaces;  spaces whose duals have the Radon-Nikodym property;   (\cite{BD}, \cite{EBPA}, \cite[p. 31]{TS}).  

%\cite{CGL},  \cite{JPZ}
A Banach space $X$ has the $DP^*$-property ($DP^*P$) if all weakly compact sets in $X$ are limited \cite{CGL}. The space  $X$ has the $DP^*$-property if and only if $L(X, c_0)=CC(X, c_0)$ \cite{CGL}, \cite{IGCC}. 
%If $X$ has the $DP^*P$, then it has the $DPP$. 
If $X$ is a Schur space or if $X$ has the $DPP$ and the Grothendieck property, then $X$ has the $DP^*P$. 

 Let $1\le p\le \infty$. A Banach space $X$ has the \emph{Dunford-Pettis property of order p}  $(DPP_p)$ ($1\le p\le \infty$) if every weakly compact operator $T:X \to Y$ is $p$-convergent, for any Banach space $Y$ \cite{CS}.

%If $X$ has the $DPP_p$, then it has the $DPP_q$, if $q<p$. The $DPP_\infty$ is precisely the $DPP$, and every Banach space  has the $DPP_1$. $C(K)$ spaces and $L_1$ have the $DPP$, and thus the $DPP_p$ for all $p$. If $1<r<\infty$, then $\ell_r$ has the $DPP_p$ for $p<r^*$. If $1<r<\infty$, then $L_r(\mu)$ has the $DPP_p$ for $p<min(2,r^*)$.

Let $1\le p\le \infty$. A Banach space $X$ has the $DP^*$-\emph{property of order} $p$  ($DP^*P_p$) if all weakly-$p$-compact sets in $X$ are limited \cite{FZ}. 

%If $q>p$ and $X$ has the $DP^*P_q$, then it has the $DP^*P_p$. The $DP^*P_\infty$ is precisely the $DP^*P$ and every Banach space has the $DP^*P_1$. 
If $X$ has the  $DP^*P$, then  $X$ has the $DP^*P_p$, for all $1 \le p \le \infty$. If $X$ has the $DP^*P_p$, then $X$ has the $DPP_p$.

Let $1\le p<\infty$. A Banach space $X$ has the $p$-\emph{Gelfand-Phillips ($p$-$GP$)  property} (or is a $p$-\emph{Gelfand-Phillips space}) if every  limited weakly $p$-summable  sequence in $X$ is norm null \cite{FZ2}. 

If $1\le p<q$ and $X$ has the $q$-$GP$ property, then $X$ has the $p$-$GP$ property. If $X$ has the $GP$ property, then $X$ has the $p$-$GP$ property for any $1\le p<\infty$. Separable spaces with the $DP^*P_p$ have the $p$-$GP$ property \cite{FZ2}.  
% \ell_1 is separable and has the $DP^*P$, thus the $DP^*P_p$ for all p.
%The space $\ell_\infty$ does not have the $p$-$GP$ property for any $1\le p<\infty$ \cite{FZ2}. 

An operator $T:X\to Y$ is called \emph{limited $p$-convergent} if it carries limited weakly $p$-summable sequences in $X$ to norm null ones in $Y$ \cite{FZ2}.

The bounded subset  $A$ of $X^*$ is called an \emph{$L$-subset} of $X^*$ if each weakly null sequence $(x_n)$ in $X$ tends to $0$ uniformly on $A$.

%A    bounded subset $A$ of $X^*$ is an $L$-subset of $X^*$ if and only if $T^*(A)$ is relatively compact  whenever $Y$ is a Banach space and  $T:Y \to X$ is a  weakly precompact operator if and only if  $T^*(A)$ is relatively compact    whenever $Y$ is a Banach space and  $T:Y \to X$ is a  weakly compact operator \cite{IGMonath}, \cite{GLC}.

A bounded subset $A$ of $X^*$ (resp. of $X$) is called a  $V$-subset of $X^*$ (resp. a $V^*$-subset of $X$) provided that
$$ \sup \{ \, |x^* (x_n )|: x^* \in A \, \} \to 0 $$
$$(\text{resp.} \, \sup\{ \, |x_n^* (x) |: x \in A \,\} \to 0)$$
for each  wuc series $\sum x_n$ in $X$ (resp. $\sum x_n^*$ in $X^*$).

The Banach space $X$ has property $(V)$ (resp. $(V^*)$) if  every $V$-subset of $X^*$ (resp. $V^*$-subset of $X$)  is relatively weakly compact. The following results were  established in \cite{AP}: $C(K)$ spaces  have property $(V)$; reflexive Banach spaces have both properties $(V)$ and $(V^*)$;  $L_1$ - spaces have property $(V^*)$;   
 $X$  has property $(V)$ if  and only if every unconditionally converging operator $T$ from $X$ to any Banach space $Y$ is weakly compact.

Let $1\le p< \infty$. We say that a bounded subset $A$ of $X$ is called a \emph{weakly-p-Dunford-Pettis} set if for all  weakly $p$-summable sequences $(x_n^*)$ in $X^*$,
$$\sup_{x\in A} |x_n^*(x)|\to 0. $$

Let $1\le p< \infty$. We say that a bounded subset $A$ of $X^*$ is called a \emph{weakly-p-L}-set  if for all  weakly $p$-summable sequences $(x_n)$ in $X$,
$$\sup_{x^*\in A} |x^*(x_n)|\to 0. $$

%of $X$ (resp. $A$ $V^*$-subset of $X$ (resp. a \text{(resp.} (resp.
%$$\lim_n (\sup\{ \, |x_n^* (x) |: x \in A \,\} ) = 0$$

%In his fundamental paper \cite{AP}, Pelczy\'{n}ski introduced property $(V)$ and property $(V^*)$.  also

%A bounded subset $A$ of $X^*$ (resp. of $X$) is called a  $V$-subset of $X^*$ (resp. a $V^*$-subset of $X$) provided that
%$$ \sup \{ \, |x^* (x_n )|: x^* \in A \, \} \to 0 $$
%$$(\text{resp.} \, \sup\{ \, |x_n^* (x) |: x \in A \,\} \to 0)$$
%for each  wuc series $\sum x_n$ in $X$ (resp. $\sum x_n^*$ in $X^*$).

%The Banach space $X$ has property $(V)$ (resp. $(V^*)$) if  every $V$-subset of $X^*$ (resp. $V^*$-subset of $X$)  is relatively weakly compact. The following results were  established in \cite{AP}: $C(K)$ spaces  have property $(V)$; reflexive Banach spaces have both properties $(V)$ and $(V^*)$;  $L_1$ - spaces have property $(V^*)$;  
%and reflexive Banach spaces
%the Banach space $X$  has property $(V)$ if  and only if every unconditionally converging operator $T$ from $X$ to any Banach space $Y$ is weakly compact.

The weakly-$1$-$L$-subsets of $X^*$ are precisely the $V$-subsets and the weakly-$1$-Dunford-Pettis subsets of $X$ are precisely the $V^*$-subsets. If $p<q$, then a weakly-$q$-$L$-subset is a weakly-$p$-$L$-subset, since $\ell_p^w(X)\subseteq \ell_q^w(X)$. Similarly, a weakly-$q$-DP set is a weakly-$p$-DP set, if $p<q$.  
% if $p=1$. 

%If $p<q$, then a weakly-$q$-$L$-set is a weakly-$p$-$L$-set, since $\ell_p^w(X)\subseteq \ell_q^w(X)$. 
	
	The  Banach space $X$ has the \emph{reciprocal Dunford-Pettis}  ($RDP$) property if every completely continuous  operator $T$ from $X$ to any Banach space $Y$ is weakly compact \cite[p. 153]{Gro}. 
The  space $X$ has the $RDP$ property if and only if every $L$-subset of $X^*$ is relatively weakly compact  \cite{TL}, \cite{GLC}.   Banach spaces with property $(V)$ of Pe{\l}czy\'nski, in particular reflexive spaces and $C(K)$ spaces, have the $RDP$ property \cite{AP}. 
A Banach space $X$ does not contain $\ell_1$ if and only if  every $L$-subset of $X^*$ is relatively compact  if and only if  every DP subset of $X^*$ is relatively compact \cite{EBPA}.

 The Banach space $X $ has property $RDP^*$ if every DP subset of $X$ is relatively weakly compact \cite{BLO}. The space $X$ has $RDP^*$ whenever $X$ has property $(V^*)$ or $X$ is weakly sequentially complete \cite{BLO}. Also, $X^*$ has $RDP^*$ whenever $X$ has property $(V)$.

Let $1\le p< \infty$. 
We say that the space $X$ has the \emph{reciprocal Dunford-Pettis  property  of order p} or  $RDP_p$ (resp. the \emph{weak reciprocal Dunford-Pettis  property of order p} or $wRDP_p$) if every weakly-$p$-$L$-subset of $X^*$ is relatively weakly compact (resp. weakly precompact). 
%and that  $X$ has the \emph{weak reciprocal Dunford-Pettis of order p} ($wRDP_p$)  property if every weakly-p-L-subset of $X^*$ is weakly precompact. We say that $X$ has  \emph{property} $RDP^*_p$   if every weakly-p-Dunford Pettis subset of $X$ is relatively weakly compact.
%has the \emph{reciprocal Dunford-Pettis of order p} ($RDP_p$)  property if every p-convergent  operator $T$ from $X$ to any Banach space $Y$ is weakly compact and that  $X$ has the \emph{weak reciprocal Dunford-Pettis of order p} ($wRDP_p$)  property if every weakly-p-L-subset of $X^*$ is weakly precompact. 

%($1\le p< \infty$)
If $X$ has the $RDP_p$ property, then $X$ has the $RDP$ property (since any $L$-subset of $X^*$ is a weakly-$p$-$L$-set). 
%(since every completely continuous operator on $X$ is p-convergent, thus weakly compact). 
If $p<q$ and $X$ has the  $RDP_p$ property, then $X$ has the  $RDP_q$ property.

We say that $X$ has  property $RDP^*_p$   if every weakly-$p$-Dunford Pettis subset of $X$ is relatively weakly compact.

%Let $1\le p< \infty$. 
 If $p<q$ and $X$ has property $RDP^*_p$, then $X$ has property $RDP^*_q$. 
%since any weakly-q-DP subset of $X$ is a weakly-p-DP set
If $X$ has property $RDP^*_p$, then $X$ has property $RDP^*$  (since every DP subset of $X$ is a weakly-$p$-Dunford Pettis set). 
Note that $c_0$ does not have property $RDP_p^*$, since it does not have property $RDP^*$.  Consequently, if $X$ has property $RDP^*_p$, then $X$ does not contain a copy of $c_0$.

%%%%%%%%%%%%%%%%%%%%%%%%%%%%%%%%%%%%%%%%%%%%%%%%%%%

\section{The $p$-Gelfand Phillips property in spaces of  operators}

In the following we give   sufficient conditions for the $p$-$GP$ property  of some spaces of operators in terms of the limited $p$-convergence of the evaluation operators. 

%\cite{Dr}

We refer to \cite{DE} for the following two facts:

(A) A sequence $(x_n)$ is limited if and only if $x_n^*(x_n)\to 0$ for each $w^*$-null sequence $(x_n^*)$ in $X^*$. 

(B) A Banach space $X$ has the GP property if and only if every  limited weakly null  sequence in $X$ is norm null. 

%some necessary and

We recall the following well-known isometries (\cite[p.60]{Ru}): \\
\noindent 1) $L_{w^*}(X^*,Y)\simeq L_{w^*}(Y^*,X)$, $K_{w^*}(X^*,Y)\simeq K_{w^*}(Y^*,X)$  ($T\to T^*$)  \\
2) $W(X,Y)\simeq  L_{w^*}(X^{**},Y) $ and $K(X,Y) \simeq K_{w^*}(X^{**},Y)$ ($T\to T^{**})$.

Suppose that $X$ and $Y$ are Banach spaces and $M$ is a closed subspace of $L(X,Y)$. If
  $x\in X$ and $y^*\in Y^*$,  the evaluation operators $\phi_x:M\to Y$ and $\psi_{y^*}:M\to X^*$ are defined by
$$ \phi_x(T)=T(x), \, \psi_{y^*}(T)=T^*(y^*), T\in M. $$

%$M$ is a closed subspace of $K_{w^*}(X^*,Y)$ 
%K(X,Y)

\begin{theorem} \label{T1}
 Let $1\le p<\infty$. 

(i) Suppose that $Y$ has the $GP$ property. If $M$ is a closed subspace of $K_{w^*}(X^*,Y)$ such that  the evaluation operator
$\psi_{y^*}:M\to X$ is limited $p$-convergent  for each $y^*\in Y^*$,  then $M$ has the $p$-$GP$ property.
	
	(ii) Suppose that $Y$ has the $GP$ property.  If $M$ is a closed subspace of $K(X,Y)$ such that
  the evaluation operator $\psi_{y^*}:M\to X^*$ is limited $p$-convergent for each $y^*\in Y^*$, then $M$ has the $p$-$GP$ property.
\end{theorem}

%$K_{w^*}(X^*,Y)$

\begin{proof} 
(i) Suppose not and let $(T_n)$ be a limited weakly $p$-summable sequence in  $M$ such that $\|T_n \| = 1$ for each $n$.  Let $(x_n^*)$ be a sequence in $B_{X^*}$  so that $\|T_n (x_n^*)\| > 1/2$ for each $n$.  

Let $y^*\in Y^*$.  Since $\psi_{y^*}:M \to X$ is limited $p$-convergent, $(T_n^*(y^*))=(\psi_{y^*}(T_n))$ is  norm null. Hence $\langle T_n(x_n^*), y^*\rangle = \langle T_n^*(y^*), x_n^*  \rangle\le \|T_n^*(y^*)\|\to 0$. Thus 
$(T_n(x_n^*))$ is weakly null.

Let $(y_n^*)$ be a $w^*$-null sequence in $Y^*$. 
Let $T \in K_{w^*}(X^*,Y) $.   Since $T^*$ is $w^*$-norm sequentially continuous,
$$\langle x_n^* \otimes y_n^* , T\rangle \leq \|T^*(y_n^*) \|\to 0.$$ 
Thus $(x_n^* \otimes y_n^*)$ is $w^*$- null in $(K_{w^*}(X^*,Y) )^*$. Since $(T_n)$ is limited in $K_{w^*}(X^*,Y)$, 
$$ \langle x_n^* \otimes y_n^* , T_n\rangle = \langle T_n(x_n^*) , y_n^*  \rangle\to 0. $$ 
Thus $(T_n(x_n^*))$ is limited.  
 %Thus $ \|\psi_{y_n^*}(T)\|=\|T^*(y_n^*)\| \to 0$, and $(\psi_{y_n^*})$ is a pointwise norm null sequence of operators. Let $H=\{T_n\}$. Since $H$ is limited,  $(\psi_{y_n^*})$ converges uniformly on $H$ 
%\cite[Proposition 1.1.2, p. 23]{TS}, i.e.,
%$$\sup \{ \|\psi_{y_n^*}(T)\|:T\in H\}=  \sup \{ \|T^*(y_n^*)\| :T\in H\} \to 0.$$ 
%It follows that 
%\begin{align*}
%\sup \{ |\langle T(x^*), y_n^*\rangle| : T\in H, x^*\in B_{X^*}\}
 %&= \sup \{ |\langle x^*, T^*(y_n^*)\rangle| : T\in H, x^*\in B_{X^*} \} \\
%&=  \sup \{ \|T^*(y_n^*)\| :T\in H\} \to 0.
%\end{align*}
%Hence $H(B_{X^*})$, and thus  $(T_n(x_n^*))$, is limited.
 Then $(T_n(x_n^*))$ is norm null, since $Y$ has the $GP$ property. This contradiction concludes the proof. 

(ii) Apply (i) and the isometry $K(X,Y)\simeq K_{w^*}(X^{**},Y)$. 
\end{proof}

We note that if $X$ has the $p$-$GP$ property, then any operator $T:E\to X$ is 
limited $p$-convergent. Indeed, if $(x_n)$ is limited weakly $p$-summable, then 
$(T(x_n))$ is limited weakly $p$-summable, and thus norm null. Thus, if $X$  has the $p$-$GP$ property and $M$ is a closed subspace of $K_{w^*}(X^*,Y)$, then the evaluation operator  $\psi_{y^*}:M \to X$ is limited $p$-convergent for each $y^*\in Y^*$.

\begin{corollary} \label{Coro1}
 Let $1\le p<\infty$. 

(i) Suppose $X$  has the $p$-$GP$ property and  $Y$ has the $GP$ property. If $M$ is a closed subspace of $K_{w^*}(X^*,Y)$, then $M$ has the $p$-$GP$ property.

(ii)  Suppose $Y$  has the $p$-$GP$ property and  $X$ has the $GP$ property. If $M$ is a closed subspace of $K_{w^*}(X^*,Y)$, then $M$ has the $p$-$GP$ property.
\end{corollary}

\begin{proof}
(i) Since $X$  has the $p$-$GP$ property, $\psi_{y^*}:M \to X$ is limited $p$-convergent for each $y^*\in Y^*$.  Apply Theorem \ref{T1}. 

(ii) Apply  (i) and the isometry $K_{w^*}(X^{*},Y)\simeq K_{w^*}(Y^{*},X)$. 
\end{proof}

\begin{corollary} \label{Cinj}
 Let $1\le p<\infty$. Suppose $X$  has the $p$-$GP$ property and  $Y$ has the $GP$ property (or $X$ has the $GP$ property and  $Y$  has the $p$-$GP$ property). Then $X\otimes_\epsilon Y$ has the $p$-$GP$ property.
\end{corollary} 

\begin{proof}
The space $X\otimes_\epsilon Y$ can be embedded into the space $K_{w^*}(X^*,Y)$, by identifying $x\otimes y$ with the rank one operator $x^*\to \langle x^*,x\rangle \, y$. Apply Corollary \ref{Coro1}. 
 %By Corollary \ref{Coro1},  $K_{w^*}(X^*,Y)$ has the $p$-$GP$ property. Hence $X\otimes_\epsilon Y$  has the $p$-$GP$ property, since  the $p$-$GP$ property is inherited by subspaces. 
\end{proof}

\begin{corollary} \label{Coro4}
 Let $1\le p<\infty$.
Suppose $X^*$  has the $p$-$GP$ property and  $Y$ has the $GP$ property (or $X^*$  has the $GP$ property and  $Y$ has the $p$-$GP$ property). If $M$ is a closed subspace of $K(X,Y)$, then $M$ has the $p$-$GP$ property.
\end{corollary} 

\begin{proof}
Apply Corollary \ref{Coro1}  and the isometry $K(X,Y)\simeq K_{w^*}(X^{**},Y)$. 
\end{proof}

\begin{corollary}
 Let $1\le p<\infty$. If $X$ has the $p$-$GP$ property, then so has $\ell_1[X]$, the space of all unconditionally convergent series $\sum x_n$ in $X$ equipped with the norm $\|(x_n)\|=\sup\{ \sum |x^*(x_n)|:x^*\in B_{X^*} \}$. 
\end{corollary}

\begin{proof}
It is known that $\ell_1[X]$ is isometrically isomorphic to $K(c_0,X)$ \cite{EARCH}. Since $c_0^*\simeq \ell_1$ has the $GP$ property, Corollary \ref{Coro4} gives the conclusion. 
\end{proof}

\begin{corollary}
 Let $1\le p<\infty$. If $\mu$ is a finite measure and  $X$ has the $p$-$GP$ property, then so has $L_1(\mu)\otimes_\epsilon X $. 
\end{corollary}

\begin{proof}
The space $L_1(\mu)$ where $\mu$ is a finite measure, has the $GP$ property \cite{DE}. It is known that  $L_1(\mu)\otimes_\epsilon X \simeq  K_{w^*}(X^*, L_1(\mu))$   \cite[Theorem 5,  p. 224]{DU}. By Corollary \ref{Coro1},  this space  has the $p$-$GP$ property. 
\end{proof}

\begin{corollary}
 Let $1\le p<\infty$. If $X$ has the $p$-$GP$ property, then so has $c_0(X)$, the Banach space of sequences in $X$ that converge to zero equipped  with the norm $\|(x_n)\|=\sup_n \|x_n\|$. 
\end{corollary}

\begin{proof}
The space $c_0$ has the $GP$ property \cite{BD}. It is known that  $c_0 \otimes_\epsilon  X\simeq c_0(X)$ \cite[p. 47]{Ryan}. Then $c_0 \otimes_\epsilon  X$ has  the $p$-$GP$ property, by Corollary \ref{Cinj}. 
\end{proof}

%The space $c_0$ has the Gelfand-Phillips property (\cite{BD}). Suppose $X$ has   property $BD$. It is known that  $c_0 \otimes_\epsilon  X\simeq c_0(X)$, the Banach space of sequences in $X$ that converge to zero, with the norm $\|(x_n)\|=\sup_n \|x_n\|$ (\cite[p. 47]{Ryan}).  Then $c_0 \otimes_\epsilon  X$ has   property $BD$, by Corollary \ref{C3}. 

\begin{theorem} \label{T2}
 Let $1\le p<\infty$.  Let  $X$ and $Y$ be Banach spaces.

(i) Let   $M$ be a closed subspace of $L(X,Y)$ such that  the evaluation operator $\psi_{y^*}:M\to X^*$ is limited $p$-convergent for each $y^*\in Y^*$.  If $M$ does not have the $p$-$GP$ property,  then there is a separable subspace $Y_0$ of $Y$ and an operator $A:Y_0\to c_0$ which is not completely continuous.

 (ii)  Let $M$ be a closed subspace of $L_{w^*}(X^*,Y)$ such that  the evaluation operator $\psi_{y^*}:M\to X$ is limited $p$-convergent for each $y^*\in Y^*$. If $M$ does not have the $p$-$GP$ property,  then there is a separable subspace $Y_0$ of $Y$ and an operator $A:Y_0\to c_0$ which is not completely continuous.

%(iii) Let $M$ be a closed subspace of $L_{w^*}(X^*,Y)$ such that the evaluation operator $\phi_{x^*}: M\to Y$ is  limited $p$-convergent for each $x^*\in X^*$. If  $M$ does not have  the $p$-$GP$ property,  then there is a separable subspace $X_0$ of $X$ and an operator $A:X_0\to c_0$ which is not completely continuous. 
\end{theorem} 

\begin{proof}
(i) Suppose  $M$ is a closed subspace of $L(X,Y)$ which does not have the $p$-$GP$ property. 
Let $(T_n)$ be a limited weakly $p$-summable sequence in  $M$ such that $\|T_n \| = 1$ for each $n$.  Let $(x_n)$ be a sequence in $B_{X}$  so that $\|T_n (x_n)\| > 1/2$ for each $n$.  

Let $y^*\in Y^*$.  Since $\psi_{y^*}:M \to X^*$ is limited $p$-convergent, $(T_n^*(y^*))=(\psi_{y^*}(T_n))$ is  norm null. Then $ \langle y^*, T_n(x_n)\rangle    \le \|T_n^*(y^*)\|\to 0$. Therefore $(y_n):=(T_n(x_n))$ is weakly null in $Y$.

 By the Bessaga-Pelczynski selection principle \cite{JDSS}, we may (and do) assume that $(y_n)$ is a seminormalized weakly null basic sequence in $Y$.  Let $Y_0=[y_n]$ be the closed linear span of $(y_n)$ and let $(y_n^*)$ be the sequence of coefficient functionals associated with $(y_n)$. Define $A:Y_0\to c_0$ by $A(y)= (y_k^*(y))$, $y\in Y_0$. Note that $\|A(y_n)\| \ge 1$ for each $n$. Then $A$ is a bounded linear operator defined on a separable space, and $A$ is not completely continuous.

(ii)  Suppose  $M$ a closed subspace of $L_{w^*}(X^*,Y)$ which does not have the $p$-$GP$ property. Let $(T_n)$ be a limited weakly $p$-summable sequence in  $M$ such that $\|T_n \| = 1$ for each $n$.  Let $(x_n^*)$ be a sequence in $B_{X^*}$  so that $\|T_n (x_n^*)\| > 1/2$ for each $n$.  

%%Therefore $(T_n(x_n^*))$ is weakly null in $Y$. For each $n$, let $y_n= T_n(x_n^*)$. 

 Let $y^*\in Y^*$. Since $\psi_{y^*}:M\to X$ is limited $p$-convergent, $(T_n^*(y^*))=(\psi_{y^*}(T_n))$ is  norm null.
% Then $ \langle y^*, T_n(x_n^*)\rangle  = \langle T_n^*(y^*), x_n^*\rangle  \le \|T_n^*(y^*)\|\to 0$. 
Therefore $(y_n):=(T_n(x_n^*))$ is weakly null in $Y$.  Continue as in  (i).
%(iii) Suppose that $M$ is a closed subspace of $L_{w^*}(X^*,Y)\simeq L_{w^*}(Y^*,X)$ which satisfies the assumptions. Apply (ii). 
\end{proof} 

% Since $Y$ has the Schur property, $(T_n(x_n))$ is norm null, and we have a contradiction. 

%Suppose that $Y$ has the Schur property. If  $M$ is a closed subspace of $L(X,Y)$ such that  the evaluation operator $\psi_{y^*}:M\to X^*$ is limited $p$-convergent for each $y^*\in Y^*$, then $M$ has the $p$-$GP$ property.

\begin{corollary} \label{Corcc}
Let $1\le p<\infty$.

(i) Suppose $X^*$  has the $p$-$GP$ property  and  $M$ is a closed subspace of $L(X,Y)$. If $M$  does not have the $p$-$GP$ property, then  there is a separable subspace $Y_0$ of $Y$ and an operator $A:Y_0\to c_0$ which is not completely continuous. 
%and $Y$ has the Schur property, then $L(X,Y)$ has the $p$-$GP$ property. 

(ii) Suppose $X$  has the $p$-$GP$ property and  $M$ is a closed subspace of $L_{w^*}(X^*,Y)$. If $M$  does not have the $p$-$GP$ property, then  there is a separable subspace $Y_0$ of $Y$ and an operator $A:Y_0\to c_0$ which is not completely continuous.  

%(iii) Suppose $Y$  has the $p$-$GP$ property and  $M$ is a closed subspace of $L_{w^*}(X^*,Y)$. If $M$  does not have the $p$-$GP$ property, then  there is a separable subspace $X_0$ of $Y$ and an operator $A:X_0\to c_0$ which is not completely continuous.  
\end{corollary}

\begin{proof}
 (i) Since $X^*$  has the $p$-$GP$ property, $\psi_{y^*}:M\to X^*$ is limited $p$-convergent for each $y^*\in Y^*$.  Apply Theorem \ref{T2}. 

(ii) Since $X$  has the $p$-$GP$ property, $\psi_{y^*}:M\to X$ is limited $p$-convergent for each $y^*\in Y^*$.  Apply Theorem \ref{T2}. 
%(iii)  Apply (ii) and the isometry $L_{w^*}(X^{*},Y)\simeq L_{w^*}(Y^{*},X)$. 
%Since $Y$  has the $p$-$GP$ property, $\phi_{x^*}:M\to Y$ is limited $p$-convergent for each $x^*\in X^*$.  Apply Theorem \ref{T2}. 
\end{proof}

%Since $X^*$  has the $p$-$GP$ property, $\psi_{y^*}:L(X,Y)\to X^*$ is limited $p$-convergent for each $y^*\in Y^*$.  Apply Theorem \ref{T2}. 

\begin{corollary}
Let $1\le p<\infty$.

(i) Suppose $X^*$  has the $p$-$GP$ property and $Y$ has the Schur property. If  $M$ is a closed subspace of $L(X,Y)$, then $M$ has the $p$-$GP$ property.

(ii)  Suppose $X$  has the $p$-$GP$ property and $Y$ has the Schur property. If  $M$ is a closed subspace of $L_{w^*}(X^*,Y)=K_{w^*}(X^*,Y)$, then $M$ has the $p$-$GP$ property. 

(iii) Suppose $Y$  has the $p$-$GP$ property and $X$ has the Schur property. If  $M$ is a closed subspace of $L_{w^*}(X^*,Y)=K_{w^*}(X^*,Y)$, then $M$ has the $p$-$GP$ property. 
\end{corollary}

\begin{proof}
 (i) Suppose that $M$ does not have the $p$-$GP$ property. By Corollary \ref{Corcc}, there is a non-completely continuous operator defined on a closed linear subspace $Y_0$ of $Y$.  This is a contradiction since $Y$  has the Schur property. 

(ii) 
%Suppose that  $X$ has the  $p$-$GP$ property  and  $Y$ has the Schur property. 
Let $T\in L_{w^*}(X^*,Y)$. Since $T$ is weakly compact and  $Y$ has the Schur property, $T$  is compact.  
Continue as above. 

(iii) It follows from (ii) and the isometries 1) on page 5.
\end{proof}

It is known that $\ell_\infty$ does not have the $GP$ property. Further, $\ell_\infty$ does not have the $p$-$GP$ property  for any $1\le p< \infty$ \cite{FZ2}.

%This is similar to the proof of Thm 7, the BD prop in spaces of compact operators

\begin{theorem} \label{Tnec}
%Let $X$ be an infinite dimensional Banach space and let  $1\le p<\infty$.  
Let  $1\le p<\infty$. Suppose that $L_{w^*}(X^*,Y)$ has  property $GP$ (resp. $p$-$GP$).  Then $X$ and $Y$ have  property $GP$ (resp. $p$-$GP$)  and either  $\ell_2 \not \hookrightarrow X$ or $\ell_2 \not \hookrightarrow Y$.
%$\ell_1\not \hookrightarrow X$ or $\ell_2 \not \hookrightarrow Y$.
If moreover $Y$ is a dual space $Z^*$, the condition  $\ell_2 \not \hookrightarrow Y$ implies $\ell_1\not \hookrightarrow Z$.
\end{theorem}

% Suppose $\ell_1 \hookrightarrow X$ and   $\ell_2 \hookrightarrow Y$.  Then 
% $c_0  \hookrightarrow K(X,Y)$ by \cite[Theorem 20]{GLBE}. 
 %Since $c_0 \hookrightarrow L(X,Y)$, $\ell_\infty \hookrightarrow L(X,Y)$  \cite{PLStudia}, \cite{GLBE}.

\begin{proof} 
We only prove the result for the $p$-$GP$ property. The other proof is similar. 
Suppose that  $L_{w^*}(X^*,Y)$ has  property $p$-$GP$. Then $X$ and $Y$ have property $p$-$GP$, since property $p$-$GP$ is inherited by closed subspaces.  
 Suppose $\ell_2 \hookrightarrow X$ and   $\ell_2 \hookrightarrow Y$.  Then 
 $c_0 \hookrightarrow K_{w^*}(X^*,Y)$ by \cite[Theorem 20]{GLBE}. 
 Since $c_0 \hookrightarrow L_{w^*}(X^*,Y)$ and $X$ and $Y$ do not have the Schur property, 
%since they contain \ell_2
$\ell_\infty \hookrightarrow L_{w^*}(X^*,Y)$ by \cite[Corollary 2]{GLBE}. 
This contradiction proves the first assertion.

Now suppose $Y=Z^*$ and  $\ell_1 \hookrightarrow Z$. Then $L_1 \hookrightarrow Z^*$ \cite[p. 212]{JDSS}.  Also, the Rademacher functions span $\ell_2$ inside of $L_1$,  hence $\ell_2 \hookrightarrow Z^*$. 
\end{proof}

\begin{corollary}
Let  $1\le p<\infty$. Suppose that $W(X,Y)$ has  property $GP$ (resp. $p$-$GP$).  Then $X^*$ and $Y$ have property  $GP$ (resp. $p$-$GP$)  and either  $\ell_1\not \hookrightarrow X$ or $\ell_2 \not \hookrightarrow Y$. If moreover $Y$ is a dual space $Z^*$, the condition  $\ell_2 \not \hookrightarrow Y$ implies $\ell_1\not \hookrightarrow Z$.
\end{corollary}

\begin{proof} 
Apply Theorem \ref{Tnec} and the isometries 2) on page 5.
\end{proof}

\section{Weakly-$p$-$L$-sets and weakly-$p$-Dunford-Pettis sets}

The following result gives a characterization of $p$-convergent operators.
The case  $1<p<\infty$ of the following result  \cite[p. 45]{CS} appeared   with no proof. We include a proof for the convenience of the reader. 

%\noindent
\begin{proposition} \label{pconv} 
Let $1\le p< \infty$. An operator $T:X\to Y$ is p-convergent if and only if for any operator $S:\ell_{p^*} \to X$ if $1<p<\infty$ (resp. $S:c_0\to X$ if  $p=1$), the operator $TS$ is compact. 
\end{proposition}

%Let $S(e_n)=x_n$. Since $(e_n)$ is weakly-p-summable  in $\ell_{p^*}$, $(x_n)$ is weakly-p-summable  in $X$, and 
 %$\|T(x_n)\|=\|TS(e_n)\|\to 0$. Note that $TS(B_{\ell_{p^*}})$ is a  subset of the $p^*$-convex hull of $(TS(e_n))$, which is relatively compact.

 \begin{proof} Suppose $T:X\to Y$ is $p$-convergent. Let $1<p<\infty$ and let $S:\ell_{p^*} \to X$ be an operator. Then $S$ is weakly-$p$-compact, since $\ell_{p^*}\in W_p$ \cite{CS}. 
 Hence $TS:\ell_{p^*}\to Y$ is compact.
%, since $T:X\to Y$ is p-convergent. 
Let $p=1$ and let  $S:c_0\to X $ be an operator. Then $TS:c_0\to Y$ is unconditionally converging, and $\sum TS(e_n)$ is  unconditionally convergent. Hence $TS$ is compact (\cite[Theorem 1.9, p. 9]{DJT}, \cite[p. 113]{JDSS}). 

Conversely, let $(x_n)$ be weakly $p$-summable in $X$.  Then $\ell_{p}^w (X)\simeq L(\ell_{p^*}, X)$ if $1<p < \infty $ (resp.  $\ell_{p}^w (X) \simeq L(c_0, X)$ if $p=1$) (\cite{GroSP}, \cite[Proposition 2.2, p. 36]{DJT}). 
Let  $S:\ell_{p^*} \to X$ if $1<p<\infty$ (resp. $S:c_0\to X$ if  $p=1$) be an operator such that $S(e_n)=x_n$. Since $TS$ is compact, $\|T(x_n)\|=\|TS(e_n)\|\to 0$, and thus $T$ is $p$-convergent. 
\end{proof}

\begin{theorem} \label{TwpDP}
Let $1\le p< \infty$. Let $T:Y\to X$ be an operator. The following are equivalent: 

1. (i)  $T(B_Y)$ is a weakly-$p$-Dunford-Pettis set.
% $T$ is weak reciprocal Dunford-Pettis (resp.  weak limited). 

(ii) $T^*:X^*\to Y^*$ is $p$-convergent.  

(iii) If $1<p<\infty$ and  $S:\ell_{p^*} \to X^*$ (resp. $p=1$ and $S:c_0\to X^*$) is an operator, then  $T^*S$ is  compact. \\

2. (i)  $T^*(B_{X^*})$ is a weakly-$p$-$L$-set.
% $T$ is weak reciprocal Dunford-Pettis (resp.  weak limited). 

(ii) $T$ is $p$-convergent.  

(iii) If $1<p< \infty$ and  $S:\ell_{p^*} \to Y$ (resp. $p=1$ and $S:c_0\to Y$) is an operator, then  $TS$ is  compact. 
\end{theorem}

\begin{proof} We only prove case 1. The other proof is similar. 

$(i) \Leftrightarrow (ii)$ 
%Let $(x_n^*)$ be a weakly $p$-summable sequence in $X^*$.
This follows directly from the equality
$$ \|T^*(x_n^*)\|= \sup\{ |\langle T(y), x_n^*\rangle |: y\in B_Y\},$$ 
for all weakly $p$-summable sequences $(x_n^*)$ in $X^*$.

%$(i)  \Rightarrow (ii)$  Let $(x_n^*)$ be a weakly-p-summable sequence in $X^*$. Suppose by contradiction that $\|T^*(x_n^*)\|>\epsilon$, for some $\epsilon>0$. Choose $(y_n)$ in $B_Y$ such that $\langle T^*(x_n^*), y_n\rangle >\epsilon$. Thus $\langle x_n^*, T(y_n)\rangle >\epsilon$. This is a contradiction, since $T(B_Y)$ is a weakly-p-Dunford-Pettis set.

%$(ii)  \Rightarrow (i)$ Suppose $(x_n^*)$ is a weakly-p-summable sequence in $X^*$. Let $(y_n)$ in $B_Y$. Then 
%$$ \langle T^*(x_n^*), y_n\rangle= \langle x_n^*, T(y_n)\rangle \le \|T^*(x_n^*)\| \to 0, $$
%and thus $T(B_Y)$ is a weakly-p-Dunford-Pettis set.

% a sequence Lemma
$(ii)  \Leftrightarrow (iii)$ by Proposition \ref{pconv}.
\end{proof}

In the next theorem we give elementary operator theoretic characterizations of weak precompactness, relative weak compactness, and relative norm compactness for weakly-$p$-Dunford-Pettis sets.
% and weakly-p-L-sets. 

\begin{theorem}  \label{TRDP^*_p}
Let $1\le p< \infty$. Let $X$ be a Banach space. The following statements are equivalent: 

(i) For every Banach space $Y$, if   $T:Y\to X$ is an   operator such that $T^*:X^*\to Y^*$ is $p$-convergent, then $T$   is weakly precompact (weakly compact, resp. compact). 
 
(ii)  same as (i) with $Y=\ell_1$.

(iii) Every  weakly-$p$-Dunford-Pettis  subset of $X$ is weakly precompact (relatively weakly compact, resp. relatively compact).
\end{theorem}

% for weakly-p-Dunford-Pettis sets.

\begin{proof} We will show that $(i)  \Rightarrow (ii)\Rightarrow (iii) \Rightarrow (i)$ in the relatively weakly compact case. The arguments for the remaining implications of the theorem follow the same pattern.

$(i)  \Rightarrow (ii)$ is obvious.

 $(ii)  \Rightarrow (iii)$ Let $K$ be a weakly-$p$-Dunford-Pettis subset of $X$ and let $(x_n)$ be a sequence in $K$.  Define $T: \ell_1 \to X$ by $T(b)=\sum b_i \, x_i$. Note that $T^*:X^*\to \ell_\infty$, 
$T^*(x^*)=(x^*(x_i))$. Suppose $(x_n^*)$ is a weakly $p$-summable sequence in $X^*$. 
%Then
Since $K$ is a weakly-$p$-Dunford-Pettis set, 
$$\|T^*(x_n^*)\|=\sup_i|x_n^*(x_i)|\to 0. $$
Therefore $T^*$ is $p$-convergent and   thus  $T$ is weakly compact. Let $(e_n^*)$ be the unit basis of $\ell_1$. Then $(T(e_n^*))=(x_n)$ has a weakly convergent subsequence.

%Suppose $(x_n^*)$ is a   weakly-p-summable sequence in $X^*$. Let $S:X\to c_0$ be defined by $S(x)=(x_n^*(x))$.  Then  $S^*(e_n^*)=x_n^*$ and  $S$ is completely continuous (since $(x_n^*)$ is an $L$-sequence).  The set $\{ ST(e_n^*): n \in \mathbb{N}\}=\{S(x_n): n \in \mathbb{N}\}$ is relatively weakly compact,  since $K$ is a  weakly-p-Dunford-Pettis set.  Note that  $ST(B_{\ell_1} )$ is contained in the closed and absolutely convex hull of $\{ S(x_n) : n \in \mathbb{N}\}$, which is relatively weakly compact  (\cite[p. 51]{DU}). Then $ST$ is weakly compact. Hence  $(T^*(x_n^*))=(T^*S^*(e_n^*))$ is weakly null, and  $T$ is a weak reciprocal Dunford-Pettis operator (by Theorem \ref{T1}). 

%Note that $T(B_{\ell_1})$ is a subset of the closed convex hull of $\{x_i:i\in\mathbb{N} \}$, a weakly-p-Dunford-Pettis set. Hence $T(B_{\ell_1})$ is a weakly-p-Dunford-Pettis set, and thus $T^*$ is p-convergent (by Theorem \ref{TwpDP}).

$(iii)  \Rightarrow (i)$ Let $T:Y\to X$ be an operator such that $T^*:X^*\to Y^*$ is $p$-convergent. Then $T(B_{Y})$ is a weakly-$p$-Dunford-Pettis set, thus relatively weakly compact. Hence $T$ is weakly compact. 
%Since  $T(B_{Y})$ is a $RDPw$ subset of $X$
\end{proof}

%\begin{corollary}
%Let $1\le p< \infty$.  If $X^*$ has property $(V)$, then $X$ has property $RDP_p^*$. 

%(ii) If $X$ has property $(V^*)$, then $X$ has property $RDP_p^*$. 

%(iii) The space $\ell_1$ has property $RDP_p^*$.
%\end{corollary}

%\begin{proof}
 %Suppose $Y$ is a Banach space and $T:Y\to X$ is an   operator such that $T^*:X^*\to Y^*$ is p-convergent. Then $%T^*$ is unconditionally converging. Since  $X^*$ has property $(V)$, $T^*$ is weakly compact \cite{AP}. Hence $T$ %is weakly compact. 
%\end{proof} 

% If X has the DPP, then it has the DPP_p. Apply (i)
%almost weakly compact

%The Odell-Rosenthal-Stegall theorem used completely continuous maps with range in $L_1$ to identify  weakly precompact operators.
Odell, Rosenthal, and Stegall \cite[p. 377]{HR}, showed that an operator $T:Y \to X$ is  weakly precompact if $L T: Y \to L_1$ is compact whenever $L : X \to L_1$ is a completely continuous map. 
%(The converse of this statement is straightforward.) 
Maps  with range in $\ell_p$ ($1< p< \infty$) can also be employed to identify such operators.

\begin{theorem} \label{pconvadjwpc}
Let $2< p< \infty$. 

(i) If $T:Y\to X$ is an operator such that $J T : Y \to \ell_p$ is compact for all  operators $J: X \to \ell_p$, then $T$ is weakly precompact.
%$T^*:X^*\to Y^*$ is p-convergent

(ii) If $T:Y\to X$ is an operator such that $T^*:X^*\to Y^*$ is $p$-convergent, then $T$ is weakly precompact.
\end{theorem}

 %Let $j:S\to c_0$ be the natural inclusion.  Since $j$ naturally factors through $\ell_2$, $j$ is absolutely 2-summing\cite{Pietsch}, Satz 2.  Now use the fact that all closed linear subspaces of an $L_2$ - space are complemented and the constructions in Theorem 2 of \cite{Pietsch} ($U^{0} = (B_{X^*} , w^*)$) and on pp. 60 -   61 of \cite{JDSS} to obtain a weakly compact operator $J: X \to c_0$ which extends $j$.  Then $JT : Y \to c_0$ is compact.
 % But $JT (y_n) = j(e_n^*) = e_n$, and $(e_n)$ is certainly not relatively compact in $c_0$.

\begin{proof} 
(i) Let  $T$ be an operator as in the statement of the theorem. Suppose (by way of contradiction) that $(y_n)$ is a sequence in $B_Y$ and $(T(y_n))$ has no weakly Cauchy subsequence.  By Rosenthal's $\ell_1$ - theorem, we can assume that $(T(y_n)) \sim (e_n^*)$, where $(e_n^*)$ is the unit vector basis of $\ell_1$.  Let $S = [ T(y_n )]$, an isomorph of $\ell_1$. Let $j:S\to \ell_p$ be the natural inclusion.
The canonical injection $j_1:\ell_1\to \ell_2$ is (absolutely) $1$-summing \cite[Example  6.17, p. 145]{Ryan}, and thus it is $2$-summing. Since $j$ naturally factors through $\ell_2$, $j$ is $2$-summing. 
Now use the fact that all closed linear subspaces of an $L_2$ - space are complemented and the constructions on p. 60 - 61 of \cite{JDSS} to obtain an operator $J: X \to \ell_p$ which extends $j$ (see also \cite[Proposition 6.24, p. 150]{Ryan}).  
%Let $J: X \to \ell_p$ be an extension of  $j$. 
 Then $JT : Y \to \ell_p$ is compact.
%By Corollary \ref{CwDP}, $JT : Y \to \ell_p$ is compact.
  But $JT (y_n) = j(e_n^*) = e_n$, and $(e_n)$ is  not relatively compact in $\ell_p$.
	
	(ii) Let $T:Y\to X$ is an operator such that $T^*:X^*\to Y^*$ is $p$-convergent and let $J: X \to \ell_p$
	be an operator. By Theorem \ref{TwpDP}, $T^*J^*:\ell_{p^*} \to Y^*$ is compact.
 Then $JT : Y \to \ell_p$ is compact. Apply (i). 
\end{proof}

The weak precompactness of a DP set is well known;
%and has been used by many authors; 
e.g., see \cite{KA}, \cite[p. 377]{HR}.
%, , \cite{BLO}, \cite{DE} and \cite{EARCH}.  
We obtain an analogous result for weakly-$p$-Dunford-Pettis sets ($2< p< \infty$).
% ($1< p< \infty$).

\begin{corollary} \label{Cwpc}
Let $2< p< \infty$. 
Every weakly-$p$-Dunford-Pettis subset of $X$ is weakly precompact.
\end{corollary}

\begin{proof}
 Let $T:Y\to X$ be  an   operator such that $T^*:X^*\to Y^*$ is $p$-convergent.  By Theorem \ref{pconvadjwpc}, $T$ is weakly precompact. Apply Theorem \ref{TRDP^*_p}. 
%Apply Theorem \ref{TRDP^*_p} and Theorem \ref{pconvadjwpc}. 
\end{proof}

A Banach space $X$ is called \emph{weakly sequentially complete} if every weakly Cauchy sequence in $X$ is weakly convergent. 
	
%$1< p< \infty$.

\begin{corollary} \label{Cwsc}
Let $2< p< \infty$. 

(i) If $X$ is weakly sequentially complete, then $X$ has property $RDP^*_p$. 

(ii) If $X$ has the Schur property, then every weakly-p-Dunford-Pettis subset of $X$ is relatively compact.

\end{corollary}

%\begin{proof}     Let $A$ be a weakly-p-Dunford-Pettis subset of $X$. By Corollary \ref{Cwpc}, $A$ is weakly %precompact. (i) Since $X$ is weakly sequentially complete, $A$ is relatively weakly compact. 
%(ii) Since $X$ has the Schur property, $A$ is relatively  compact.  
%\end{proof}

The following theorem gives a characterization of weakly-$p$-$L$-sets. 
For any $p$, $1\le p<\infty$, a bounded subset $K$ of $\ell_p$ is relatively compact if and only if $\lim_n \sum_{i=n}^\infty|k_i|^p=0$, uniformly for $k \in K$ \cite[p. 6]{JDSS}.

%We note that an operator $T:X\to Y$ is p-convergent if and only if $T$ maps weakly-p-Cauchy sequences to norm convergent sequences. same as the obs before cor 3

\begin{theorem} \label{pLset}
Let $1<p<\infty$. Suppose that $A$ is a  bounded subset of $X^*$. The following are equivalent:

(i) $A$ is a weakly-$p$-$L$-subset of $X^*$. 

%weakly p-precompact
 
(ii) $T^*(A)$ is relatively compact  whenever $Y$ is a Banach space and  $T:Y \to X$ is a  weakly $p$-precompact operator.

%(iii) $T^*(A)$ is relatively compact  whenever $Y$ is a Banach space and  $T:Y \to X$ is a  weakly p-compact operator.

(iii)  $T^*(A)$ is relatively compact  whenever $Y\in WPC_p$ and $T:Y \to X$ is an   operator.

(iv) $T^*(A)$ is relatively compact  whenever  $T:\ell_{p^*} \to X$ is an   operator.

(v) If $(x_n)$ is a weakly $p$-summable sequence in $X$ and $(x_n^*)$ is a sequence in $A$, then $\lim x_n^*(x_n)=0 $. 
\end{theorem}

%Only the nontrivial relations among the statements are detailed. 
\begin{proof} 
$(i)\Rightarrow (ii)$ Suppose that $A$ is a weakly-$p$-$L$-subset of $X^*$ and let $(x_n^*)$ be a sequence in $A$. Let $T:Y \to X$ be a weakly-$p$-precompact operator. Define $S:X\to \ell_\infty$ by 
$S(x)=(x_n^*(x))$.  
Let $(x_n)$ be a  weakly $p$-summable sequence in $X$. Since  $A$ is a weakly-$p$-$L$-subset of $X^*$,  
$\lim_n \|S(x_n)\|=\lim_n \sup_i |x_i^*(x_n)|=0 $, and thus $S$ is $p$-convergent. 
Then $ST:Y\to \ell_\infty$ is compact,  since $T$ is weakly-$p$-precompact. Let $(e_n^*)$ be the unit basis of $\ell_1$. Since $T^*S^*$ is compact, $(T^*(x_n^*))=(T^*S^*(e_n^*))$ is relatively compact.  

%Certainly
 $(ii)\Rightarrow (iii)$ If $Y\in WPC_p$, then any operator $T:Y \to X$ is weakly-$p$-precompact. 

$(iii)\Rightarrow (iv)$ Suppose $T:\ell_{p^*} \to X$ is an operator. Since $1< p^*<\infty$, 
%the identity map on $\ell_{p^*}$ is weakly-p-compact 
$\ell_{p^*} \in W_p$ \cite[Proposition 1.4]{CS}.
% Hence $T$ is  weakly-p-compact. By (iii),
Then  $T^*(A)$ is relatively compact. 
%$\ell_{p^*}\in WPC_p$ and 

$(iv)\Rightarrow (i)$ Let $(x_n)$ be a weakly $p$-summable sequence in $X$. Let  $T:\ell_{p^*} \to X$ such that 
$T(e_n)=x_n$ (\cite{GroSP},  \cite[Proposition 2.2]{DJT}). Since $T^*(A)$ is relatively compact in $\ell_p$,  
$ \sup_{x^*\in A} |\langle T^*(x^*), e_n\rangle|= \sup_{x^*\in A} |\langle x^*, x_n\rangle|\to 0$. 
%Hence $A$ is a weakly-p-L-subset of $X^*$. 

$(i)\Rightarrow (v)$ is obvious. 

% If $(x_n)$ is a weakly-p-summable sequence in $X$ and $(x_n^*)$ is a sequence in $A$, then $ |x_n^*(x_n)|\le \sup_{x^*\in A} |x^*(x_n)| \to 0 $. 

	$(v)\Rightarrow (i)$ Let  $(x_n)$ be a weakly $p$-summable sequence in $X$. Since $A$ is bounded, for every $n$ we can choose $x_n^*$ in $A$ such that 
	$\sup_{x^*\in A}  |x^*(x_n)| \le 2 |x_n^*(x_n)|$. Then $\sup_{x^*\in A}  |x^*(x_n)| \le 2 |x_n^*(x_n)|\to 0$, and $A$ is an weakly-$p$-$L$-set. 
	% see Slumprecht's diss, p. 23
\end{proof}

%By the characterization of relatively compact subsets of $\ell_p$

The following result gives a characterization of weakly-$p$-DP sets. 

\begin{corollary} \label{pDPset}
Let $1<p<\infty$. Suppose that $A$ is a bounded subset of a Banach space  $X$. Then the following assertions are equivalent: 

(i) $A$ is a weakly-$p$-DP set.

(ii)  $T(A)$ is relatively compact whenever $Y$ is a Banach space and $T:X\to Y$ is an operator with  weakly-$p$-precompact adjoint. 

%If $T:X\to Y$ is an operator with  weakly-p-precompact adjoint, then $T(A)$ is relatively compact for any Banach space $Y$.  

(iii)  $T(A)$ is relatively compact whenever $Y^*\in WPC_p$  and $T:X\to Y$ is an operator. 

%If $T:X\to Y$ is an operator with  weakly-p-compact adjoint, then $T(A)$ is relatively compact for any Banach space $Y$. 

(iv)  $T(A)$ is relatively compact whenever $T:X\to \ell_p$ is an operator. 

(v) If $(x_n^*)$ is a weakly $p$-summable sequence in $X^*$ and $(x_n)$ is a sequence in $A$, then $\lim x_n^*(x_n)=0 $. 
\end{corollary} 

%Only the nontrivial relations among the statements are detailed. 

\begin{proof} 
$(i)\Rightarrow (ii)$  Let $T:X\to Y$ be an operator such that  $T^*:Y^*\to X^*$ is weakly-$p$-precompact. Since  $A$ is a weakly-$p$-DP subset of $X$, $A$ is an weakly-$p$-$L$-subset of $X^{**}$.
 By Theorem \ref{pLset}, $T^{**}(A)$ is relatively compact. Hence $T(A)$ is relatively compact.

$(ii)\Rightarrow (iii)$ If $Y^*\in WPC_p$, then any operator $T:X\to Y$  has a  weakly-$p$-precompact adjoint. 

%$(ii)\Rightarrow (iii)$ and  $(iii)\Rightarrow (v)$  are obvious.

% $(ii)\Rightarrow (iv)$  and $(iv)\Rightarrow (v)$ are obvious. 

%$(v)\Rightarrow (i)$ by \cite[Theorem 1]{KA}.

$(iii)\Rightarrow (iv)$ Let $T:X\to \ell_p$ be an   operator. Since $1< p^*<\infty$, $\ell_{p^*} \in W_p$ \cite{CS}. 
%the identity map on $\ell_{p^*}$ is weakly-p-compact (\cite{CS}). Then $T^*:\ell_{p^*}\to X^*$ is weakly-p-compact, and
 Thus $T(A)$ is relatively compact.

  $(iv)\Rightarrow (i)$ Let $(x_n^*)$ be a weakly $p$-summable sequence in $X^*$. Let  $T:\ell_{p^*} \to X^*$ such that $T(e_n)=x_n^*$ (\cite{GroSP},  \cite[Proposition 2.2]{DJT}). Let $T_1=T^*|X:X\to \ell_p$. Since $T_1(A)$ is relatively compact in $\ell_p$,  
$ \sup_{x\in A} |\langle T^*(x), e_n\rangle|= \sup_{x\in A} |\langle x, x_n^*\rangle|\to 0$. 

 $(v)\Rightarrow (i)$ Let  $(x_n^*)$ be a weakly $p$-summable sequence in $X^*$. Since $A$ is bounded, for every $n$ we can choose $x_n$ in $A$ such that 
	$\sup_{x\in A} |x_n^*(x)| \le 2 |x_n^*(x_n)|$. Then $\sup_{x^*\in A} |x^*(x_n)| \le 2 |x_n^*(x_n)|\to 0$, and $A$ is a weakly-$p$-DP set. 
\end{proof}

If $X$ is any infinite dimensional Schur space, then all bounded subsets of $X^*$ are weakly-$p$-$L$-subsets, and thus there are weakly-$p$-$L$-subsets of $X^*$ which fail to be weakly precompact.

\begin{theorem}\label{TRDP_p}
Let $1\le p< \infty$. Suppose that X is a Banach space. The following are equivalent:

(i) For every Banach space $Y$, if $T: X \to Y$ is a p-convergent operator, then $T^*:Y^*\to X^*$ is  weakly precompact (weakly compact, resp. compact). 

(ii) Same as (i) with $Y = \ell_\infty$.

(iii)  Every weakly-p-$L$-subset of $X^*$ is weakly precompact (relatively weakly compact, resp. relatively compact).
\end{theorem}

\begin{proof}
We will show that  $(i) \Rightarrow (ii) \Rightarrow (iii) \Rightarrow (i)$ in the  weakly  precompact case.  The arguments for all the remaining implications in the theorem follow the same pattern.

  $(i) \Rightarrow (ii)$ is obvious. $(ii) \Rightarrow (iii)$ 
	%Suppose that (ii) holds.  
Let $A$ be a  weakly-$p$-$L$-subset of $X^*$ and let $(x_n^*)$ be a sequence in $A$. Define $T:X\to \ell_\infty$ by 
$T(x)=(x_i^*(x)) $, $x\in X$. Suppose $(x_n)$ is weakly $p$-summable in $X$.  Since $A$ is a  weakly-$p$-$L$-subset,
$$\lim_n \|T(x_n)\|=\lim_n \sup_i |x_i^*(x_n)| =0. $$
Therefore $T$ is $p$-convergent, and thus $T^*:\ell_\infty^*\to X^*$ is weakly precompact. Let $(e_n^*)$ be the unit vector basis of $\ell_1$. 
Hence $(T^*(e_n^*))=(x_n^*)$ has a weakly Cauchy subsequence. 
%is weakly precompact. 

$(iii) \Rightarrow (i)$  Suppose that every weakly-$p$-$L$-subset of $X^*$ is weakly precompact   and let $T: X\to Y$ be a  $p$-convergent  operator. 
%Suppose  $(x_n)$ is  weakly $p$-summable in $X$. If $y^*\in B_{Y^*}$, then 
%$$\langle T^*(y^*), x_n \rangle \le \|T(x_n)\|\to 0.$$
Hence $T^*(B_{Y^*})$ is a weakly-$p$-$L$-subset of $X^*$, thus weakly precompact. Therefore $T^*$ is weakly precompact.
\end{proof}

\begin{corollary} \label{Crdpstar}
Let $1\le p< \infty$.

%(i)  If $X^*$ has property $(V)$, then $X$ has property $RDP_p^*$. 

(i) If  $X$ has property $(V^*)$, then $X$ has property $RDP_p^*$. 

(ii) If $X$ has property $(V)$, then $X$ has property $RDP_p$. 
\end{corollary}

%We only prove (i). The proof of (ii) is similar.
% (i) Suppose $Y$ is a Banach space and $T:Y\to X$ is an   operator such that $T^*:X^*\to Y^*$ is p-convergent. Then $T^*$ is unconditionally converging. Since  $X^*$ has property $(V)$, $T^*$ is weakly compact \cite{AP}. Hence $T$ is weakly compact. Apply Theorem \ref{TRDP^*_p}.

\begin{proof} 
(i) Let $A$ be a weakly-$p$-DP subset of $X$. Then $A$ is a weakly-$1$-DP subset, thus  a $V^*$-subset of $X$. Since $X$ has property $(V^*)$, $A$ is relatively weakly compact. 

(ii) Suppose $T: X \to Y$ is a $p$-convergent operator.  Then $T$ is unconditionally converging. Since  $X$ has property $(V)$, $T$ is weakly compact \cite{AP}. Hence $T^*$ is weakly compact. 
Apply Theorem \ref{TRDP_p}.
\end{proof}

\begin{corollary}\label{quotient}
Let $1\le p< \infty$. If $X$ has property $RDP_p$, then every quotient space of $X$ has property $RDP_p$. 
\end{corollary}

\begin{proof} 
Suppose that $X$ has property $RDP_p$,  $Z$ is a quotient of $X$ and $Q:X\to Z$ is a quotient map. Let $T: Z\to E$  be a $p$-convergent operator. Then $TQ:X\to E$ is $p$-convergent, and thus $(TQ)^*$ is weakly compact by   Theorem \ref{TRDP_p}. Since $Q^*$ is an isomorphism and  $Q^*T^*(B_{E^*})$ is relatively weakly compact, $T^*(B_{E^*})$ is relatively weakly compact. Apply  Theorem \ref{TRDP_p}.
\end{proof}

\begin{corollary} \label{CwDP}
Let $1< p< \infty$. Let $T:Y\to X$ be an operator. The following are equivalent: 

(i)  $T(B_Y)$ is a weakly-p-Dunford-Pettis set.

(ii) $T^*:X^*\to Y^*$ is p-convergent.  

(iii) If   $A:X \to \ell_{p}  $  is an operator, then  $AT$ is  compact. 
\end{corollary}

\begin{proof}
$(i) \Leftrightarrow (ii)$ by Theorem \ref{TwpDP}. 

$(ii)  \Rightarrow (iii)$ Let $A:X \to \ell_{p}  $ be an operator. By Theorem \ref{TwpDP}, $T^*A^*:\ell_{p^*} \to Y^*$ is compact. Then $AT : Y \to \ell_p$ is compact.

$(iii)  \Rightarrow (ii)$ 
%Let $(x_n^*)$ be a weakly-p-summable sequence in $X^*$. 
Let $S: \ell_{p^*}  \to X^*$  be an operator. Let $A=S^*|X$, $A:X\to \ell_p$. Then $A^*=S$. Since $AT$ is compact, $T^*A^*= T^*S$ is compact. Apply Theorem \ref{TwpDP}. 
%,  and  $(T^*S(e_n))=(T^*(x_n^*))$ is relatively compact. Hence $\|T^*(x_n^*)\| \to 0$, and $T^*$ is p-convergent.  
%$A^*:\ell_{p^*}\to X^*$
\end{proof}

\begin{corollary} \label{ball}
Let $1 < p< \infty$. The following are equivalent: 

1. (i)  $B_X$ is a weakly-p-Dunford-Pettis set.

%(ii) $i^*:X^*\to X^*$ is p-convergent.  

(ii)   $L(\ell_{p^*}, X^*)=K(\ell_{p^*}, X^*)$. 
%; if $p=1$, then $L(c_0, X^*)=K(c_0, X^*)$. \\

(iii) $L(X, \ell_{p})=K(X, \ell_{p})$. 

(iv) $X^*\in C_p$.\\

2. (i)  $B_{X^*}$ is a weakly-p-L-set.

%(ii) $i$ is p-convergent.  

(ii) $L(\ell_{p^*}, X)=K(\ell_{p^*}, X)$. 
%If $1<p< \infty$, then $L(\ell_{p^*}, X)=K(\ell_{p^*}, X)$; if $p=1$, then $L(c_0, X)=K(c_0, X)$. 

(iii) $X \in C_p$.
\end{corollary} 

% Note that $B_X$ is a weakly-p-Dunford-Pettis set if and only if  
\begin{proof}
We only prove case 1. The other proof is similar. 
  Apply Theorem \ref{TwpDP} and Corollary \ref{CwDP} to the identity map  $i$ on $X$. 
\end{proof}

\begin{corollary}
Let $1 < p< \infty$.

1. Suppose $L(\ell_{p^*}, X^*)=K(\ell_{p^*}, X^*)$. 
Then $X$ has the $RDP^*_p$ property if and only if $X$ is reflexive. 

2. Suppose $L(\ell_{p^*}, X)=K(\ell_{p^*}, X)$. Then $X$ has the $RDP_p$ property if and only if $X$ is reflexive. 
\end{corollary}

\begin{proof}
 We only prove case 1. The other proof is similar. Suppose $L(\ell_{p^*}, X^*)=K(\ell_{p^*}, X^*)$ and $X$ has the $RDP^*_p$ property. By Corollary \ref{ball}, $B_X$ is a weakly-$p$-Dunford-Pettis set, and thus relatively weakly compact. 
\end{proof}

\begin{corollary} \label{C27}
Let $1< p< \infty$.

(i) If $X\in WPC_p$, then every weakly-p-L-subset of $X^*$ is relatively compact.

(ii) If $X^*\in WPC_p$, then every weakly-p-Dunford Pettis subset of $X$ is relatively compact.

(iii) If $X$ is infinite dimensional and $X\in C_p$, then $X^*$ contains weakly-p-$L$-sets which are not relatively compact.  

(iv) If $X$ is infinite dimensional and $X^*\in C_p$, then $X$ contains weakly-p-Dunford Pettis sets which are not relatively compact. 
\end{corollary}

\begin{proof}
 Let $i:X\to X$ be the identity map on $X$. 

(i) Since $X\in WPC_p$, $i$ is weakly-$p$-precompact. Let $A$ be a weakly-$p$-$L$ subset of $X^*$. By Theorem \ref{pLset}, $i^*(A)=A$ is relatively compact. 

(ii) Let $A$ be a weakly-$p$-Dunford Pettis subset of $X$. Since $X^*\in WPC_p$, $i^*:X^*\to X^*$ is  weakly-$p$-precompact. By Corollary \ref{pDPset}, $i(A)=A$ is relatively compact.
 
(iii) Since $X\in C_p$ and $X$ is infinite dimensional, $i$ is $p$-convergent and not compact. Apply Theorem \ref{TRDP_p}. 

(iv) Since $X^*\in C_p$,   $i^*:X^*\to X^*$ is $p$-convergent. Further,  $i$ is not compact. Apply Theorem \ref{TRDP^*_p}.
%, since $X$ is infinite dimensional.  
\end{proof} 

% We note that if every weakly-p-L subset of $X^*$ is relatively compact, then every $L$-subset of $X^*$ is relatively compact, since every $L$-subset is a weakly-p-L set. Hence $\ell_1\not \hookrightarrow X$ \cite{EBPA}.

If $1<p<\infty$ and $1<r<p^*$, then the identity map on $\ell_r$ is $p$-convergent (\cite[Corollary 1.7]{CS}). Hence $B_{\ell_{r^*}}$ is a weakly-$p$-$L$-set which is not relatively compact. 

%the identity map on $\ell_{p^*}$ is in W_p  (\cite{CS}). Then the identity map $I_{r^*}$ on $\ell_{r^*}$ is in W_p, since $r^*>p$. Hence the adjoint $I^*_{r^*}$ on $\ell_r$ is in C_r, for all $r<p^*$. 

Let $1<p<\infty$.  Then $\ell_{p^*}\in W_p$ \cite{CS}. By Corollary \ref{C27}, every weakly-$p$-$L$-subset of $\ell_p$ is relatively compact. Further, every weakly-$p$-Dunford Pettis subset of $\ell_p$ is relatively compact.

Let $T$ be the Tsirelson's space \cite{CS}. Then   $T^*\in W_p$  and $T\in C_p$. Hence every weakly-$p$-Dunford Pettis subset of $T$ is relatively compact (by Corollary \ref{C27}). Further, $T^*$ contains weakly-$p$-$L$-sets which are not relatively compact.

%By Cor 25, If $X^*\in WPC_p$, then every weakly-$p$-Dunford Pettis subset of X is rel cpt . Then $X^*\not \in C_p$. Is the converse true? If $X^*\not \in C_p$, then $X^*\in WPC_p$?

%If $X\in WPC_p$, then every weakly-$p$-L subset of X^* is rel cpt . Then $X^**\not \in C_p$. Is the converse true? If $X^**\not \in C_p$, then $X\in WPC_p$?

% If X is infinite dimensional and X\in WPC_p, then X\not in C_p. Suppose otherwise. Let (x_n) be a sequence in B_X. Then WLOG (x_n) is w-p-Cauchy. Since X\in C_p, x_n is norm null. then B_X is relatively compact, a contradiction. 

%if $X$ has $(V)$, then $X^*$ has $(V^*)$
% if $X^*$ has $(V)$, then $X$ has $(V^*)$ \cite{AP}.

\begin{corollary} \label{C1}
Let $1\le p< \infty$. 
%If $X$ has $RDP_p$, then $X^*$ has $RDP_p^*$. 

(i) Suppose $Y$ is a closed subspace of  $X^*$ and $X$ has the $RDP_p$. Then  $Y$ has property $RDP_p^*$.  

%(ii) Suppose $F$ is a closed subspace of  $E^*$. If $F$ has  the $DPP$ and is not weakly sequentially complete, then   $E$ does not have  the $RDPP$.

(ii)  If   $Y^*$ has the $RDP_p$, then $Y$ has property $RDP_p^*$.  
\end{corollary}

%(resp. $wRDP_p$) (resp. weakly precompact)

\begin{proof} 
%Suppose $X$ has property $RDP_p$. If $(x_n)$ is a weakly p-summable sequence in $X$, then it is a weakly p-summable  sequence in $X^{**}$. Then every weakly-p-DP subset of $X^*$ is a weakly-p-L set, thus relatively weakly compact (by Theorem \ref{TRDP^*_p}). Hence $X^*$ has $RDP_p^*$.
(i) Let $K$ be a weakly-$p$-DP subset of $Y$. Then $K$ is a weakly-$p$-DP subset of $X^*$, and thus  a weakly-$p$-$L$-subset of $X^*$.  Hence $K$ is relatively weakly compact. Then  $Y$ has property $RDP_p^*$.
%, by Theorem \ref{TRDP_p}

(ii) Consider $Y$ a closed subspace of $Y^{**}$ and apply (i). 
\end{proof} 

The converse of Corollary \ref{C1} (i) is not true. 
Let $X$ be the first Bourgain-Delbaen space \cite{JBFD}. Then $X$ is an infinite dimensional $\mathcal{L}_\infty$-space with  the Schur property and $X^*$ is weakly sequentially complete. Since $X$ has the Schur property, the identity map $i$ on $X$ is completely continuous, hence $p$-convergent ($1<p< \infty$), and not weakly compact. Thus $X$ does not have property $RDP_p$. Since $X^*$ is weakly sequentially complete, $X^*$ has property $RDP^*_p$ by Corollary \ref{Cwsc}.

%(resp. $(wL)$)

\begin{corollary}\label{pconvadjoint}
Let $1\le p< \infty$.

(i) Suppose $F$ is a closed subspace of  $Z^*$ and $Z$ has the $RDP_p$. If  $T:E\to F$ is an operator such that $T^*:F^*\to E^*$ is $p$-convergent, then $T$ is weakly compact.

(ii) Suppose $F^*$ has the $RDP_p$. If  $T:E\to F$ is an operator such that $T^*:F^*\to E^*$ is $p$-convergent, then $T$ is weakly compact.
\end{corollary}

%Suppose $x\in B_E$ and $(y_n^*)$ is weakly-p-summable in $F^*$. Then $\langle T(x), y_n^*\rangle \le \|T^*(y_n^*)\|\to 0$. 

\begin{proof} (i) By Corollary \ref{C1},  $F$ has  property  $RDP^*$. 
Let  $T:E\to F$ be an operator such that $T^*:F^*\to E^*$ is $p$-convergent. By Theorem \ref{TRDP^*_p}, $T$ is weakly compact. 

 %Then  $T(B_E)$ is a weakly-p-DP subset of $F$, and thus relatively weakly compact. Then $T$ is weakly compact.

(ii)  Consider $F$ a subspace of $F^{**}$ and $Z=F^*$. Apply (i).
\end{proof}

In the following two results we need the following result.

\begin{lemma} (\cite[Theorem 2.7]{GO}) \label{GonzalezOnieva}
Let $X$ be a Banach space,  $Y$  a reflexive subspace of $X$ (resp. a  subspace  not containing copies of $\ell_1$), and $Q:X\to X/Y$ the  quotient map. Let $(x_n)$ be a bounded sequence in $X$ such that $(Q(x_n))$ is weakly convergent (resp.   weakly Cauchy). 
 Then $(x_n)$ has a weakly convergent  subsequence (resp. weakly Cauchy). 
\end{lemma}

%converges  weakly  to some $Q(x)$. 
%.relatively is weakly compact

%Let $1\le p<\infty$.p=1 done in ? 
%(i) (ii) Let  $X$ be a Banach space and $Y$ be a  subspace of $X$ not containing copies of $\ell_1$. If $X/Y$ has the $wRDP_p^*$ property, then $X$ has the $wRDP_p^*$ property.

\begin{theorem} 
Let $1< p<\infty$.
 Let $X$ be a Banach space and $Y$ be a reflexive subspace of $X$.  If $X/Y$ has the $RDP_p^*$ property, then $X$ has the $RDP_p^*$ property.
%If  every  weakly-p-DP subset of   $X/Y$ is weakly precompact, then every  weakly-p-DP subset of  $X$ is weakly precompact.
\end{theorem}

\begin{proof} 
%We only prove (i). The  proof of (ii)  is similar. 
Let  $Q:X\to X/Y$ be the quotient map.    Let $A$  be a  weakly-$p$-DP subset of $X$ and $(x_n)$ be a sequence in $A$. Then $(Q(x_n))$ is a  weakly-$p$-DP subset of $X/Y$, and thus relatively weakly compact. By passing to a subsequence, suppose $(Q(x_n))$ is weakly convergent.
By Lemma \ref{GonzalezOnieva}, $(x_n)$ has a weakly convergent subsequence. 
\end{proof}

%Let  $E$ be a Banach space and  $F$ be a  subspace of $E^*$. Let  
%$$^{\perp}F=\{ x\in E: y^*(x)=0 \, \, \text{for all} \, y^*\in F \}.$$

%(resp. $(wV)$) (resp. $(wV)$) (\emph{property $(wL)$})

%Let $1\le p<\infty$. prop V done in ?

Let  $E$ be a Banach space and  $F$ be a  subspace of $E^*$. Let  
$$^{\perp}F=\{ x\in E: y^*(x)=0 \, \, \text{for all} \, y^*\in F \}.$$

\begin{theorem} \label{c1}
Let $1< p<\infty$.
Let  $E$ be a Banach space and $F$ be a  reflexive subspace of $E^*$ (resp. a subspace not containing $\ell_1$).  If $^{\perp}F$ has property $RDP_p$ (resp. $wRDP_p$), then  $E$ has property $RDP_p$ (resp. $wRDP_p$).
\end{theorem}

\begin{proof} We will only prove the result for property $RDP_p$. The other proof is similar.
%; the proof for property  $wRDP_p$ is similar.

Suppose that $^{\perp}F$ has property $RDP_p$. Let $Q:E^*\to E^*/F$ be the quotient map and $i:E^*/F \to (^{\perp}F)^*$ be the natural surjective isomorphism \cite[Theorem 1.10.16]{M}. It is known that $iQ:E^* \to (^{\perp}F)^*$ is $w^*-w^*$ continuous, since $iQ(x^*)$ is the restriction of $x^*$ to $^{\perp}F$ \cite[Theorem 1.10.16]{M}. Then there is an operator $S:^{\perp} F \to E$   such that $iQ=S^*$.

Let $T:E\to X$ be a $p$-convergent operator. Since $^{\perp} F$ has property $RDP_p$ and  the operator $TS:^{\perp} F\to X$ is $p$-convergent, it is  weakly compact. Since $S^*T^*=iQT^*$ is weakly compact and $i$ is a surjective isomorphism, $QT^*$ is weakly compact. Let $(x_n^*)$ be a sequence in $B_{X^*}$.   By passing to a subsequence, we  can assume that $(QT^*(x_n^*))$ is weakly convergent. Hence $(T^*(x_n^*))$ has a weakly convergent subsequence by Lemma \ref{GonzalezOnieva}. Then $E$ has property $RDP_p$ by Theorem \ref{TRDP_p}.
\end{proof}

%(by Theorem \ref{TRDP_p}) it has a weakly compact adjoint weakly Cauchy  By Theorem \ref{TRDP_p},

%\newpage
%\emph{

\end{document}